\documentclass[reqno]{amsart}
\usepackage{graphicx} 
\usepackage[utf8]{inputenc}
\usepackage{amsmath}
\usepackage{amsfonts}
\usepackage{amssymb}
\usepackage{amsthm}
\usepackage{dsfont}
\usepackage[alphabetic,abbrev]{amsrefs}
\usepackage{xspace}
\usepackage{nicefrac}
\usepackage{mathrsfs}
\usepackage{stmaryrd} 
\usepackage{bm}

\usepackage{hyperref}

\usepackage{color}

\newtheorem{theorem}{Theorem}[section]
\newtheorem*{theorem*}{Theorem}
\newtheorem{lemma}[theorem]{Lemma}

\newtheorem{corollary}[theorem]{Corollary}
\newtheorem{question}{Question}

\theoremstyle{definition}
\newtheorem{definition}[theorem]{Definition}

\theoremstyle{remark}
\newtheorem{remark}[theorem]{Remark}

\newcommand{\N}{\mathbb{N}}
\newcommand{\Z}{\mathbb{Z}}

\newcommand{\C}{\mathbb{C}}
\newcommand{\T}{\mathbb{T}}
\newcommand{\D}{\mathbb{D}}
\newcommand{\PP}{\mathbb{P}}
\newcommand{\dd}{\mathrm{d}}
\newcommand{\ho}{ő\xspace}
\newcommand{\Sz}{Sz}
\newcommand{\Meas}{\mathcal{P}}
\newcommand{\MMeas}{\mathrm{Mar}(\T)}
\newcommand{\A}{\mathscr{A}}

\title{On a density problem related to a theorem of Szeg\H{o}}
\author{Chiara Paulsen}
\address{Chiara Paulsen, University of Wisconsin-Madison, Department of Mathematics, 480 Lincoln Drive, Madison, WI 53706, USA}
\email{cpaulsen3@wisc.edu}

\begin{document}
\begin{abstract}
    A classical theorem of Szeg\H{o} states that for any probability measure $\mu=w\frac{\dd\theta}{2\pi}+\mu_s$ on the unit circle the polynomials are dense in $L^2(\T,\mu)$ if and only if $\log(w)\notin L^1(\T)$. A related question asks whether the monomials with exponents in some subset $\Lambda\subseteq \N_0$ already span $L^2(\T,\mu)$ if $\log(w)\notin L^1(\T)$. A result by Olevskii and Ulanovskii gives an answer if $\mu$ belongs to a class of absolutely continuous measures. We investigate the same question for Markoff measures.
\end{abstract}
\maketitle
\section{Introduction}
Let $\mu$ be a probability measure on the complex unit circle $\T$, the latter of which we shall identify with the interval $[0,2\pi)$ in the usual way. Denote by 
\[
\dd m=\frac{\dd \theta}{2\pi}
\]
the normalized Lebesgue measure on $\T$. By the Lebesgue decomposition theorem, one can decompose $\mu$ with respect to $m$ 
\[
\dd\mu=w\,\dd m+\dd\mu_s
\]
where $\mu_s$ and $m$ are mutually singular and $w$ is the Radon-Nikodym derivative of $\mu$ and $m$. 

We call $\mu$ \textit{non-degenerate} if $|\mathrm{supp}(\mu)|=\infty$. In the following, we denote by $\mathcal{P}$ the set of all non-degenerate probability measures on $\T$.

An important theorem by Szeg\H{o} \cite{SzegoGrenander58} in the theory of orthogonal polynomials on the unit circle (OPUC) implies that for any $\mu\in \mathcal{P}$ the polynomials are dense in $L^2(\T,\mu)$ if and only if 
\begin{equation}\label{Non-Szego Condition}
    \int_0^{2\pi} \log(w(\theta))\;\mathrm{d}\theta=-\infty.
\end{equation}
In fact, this equivalence already appeared in the work of Kolmogorov \cite{Kolmogorov41} in the context of prediction theory. The connection between Kolmogorov's work on prediction theory and Szeg\ho's work on OPUC was made by Krein \cite{Krein45}. A contemporary discussion of the connection between Szeg\ho's theorem and prediction theory can be found in \cite{Bingham12}. Let us briefly outline this connection. For every Gaussian stationary stochastic process  $(X_n)_{n\in\Z}$  with zero mean, there exists $\mu\in\Meas$ such that 
\[
E[X_n\overline{X_0}]=\int_\T z^{-n}\,\dd \mu.
\]
By Kolmogorov \cite{Kolmogorov41}*{Lemma 4} there is an isomorphism between the Hilbert space $\mathcal{H}$ spanned by $\{X_n\mid n\in\Z\}$ and $L^2(\T,\mu)$ mapping $X_n$ to $z^n$. Via this isomorphism, the error of predicting $X_0$ with the knowledge of $\{X_n\mid n\leq-1\}$ is given by \cite{Bingham12}*{Theorem~3}, \cite{Kolmogorov41}*{Chapter 9}
\[
E\Big[(X_0-P_{(-\infty,-1]}X_0)^2\Big]=\exp\Big(\int_0^{2\pi}\log w\,\dd m\Big)
\]
where $P_{(-\infty,-1]}$ is the projection onto the subspace of $\mathcal{H}$ spanned by $\{X_n\mid n\leq-1\}$. 
Thus, $(X_n)_n$ is a deterministic process (that is, the past predicts the future with zero error) if and only if the associated measure $\mu$ fulfills \eqref{Non-Szego Condition}, which, in turn, is equivalent to the polynomials being dense in $L^2(\T,\mu)$.

 A measure $\mu\in\Meas$ that fulfills \eqref{Non-Szego Condition} is called \textit{non-Szeg\H{o}}. If the integral in \eqref{Non-Szego Condition} is finite then $\mu$ is called \textit{Szeg\ho}. We denote the class of all Szeg\ho measures by $\Sz$ and put $\Sz^c:=\Meas\backslash \Sz$. 

The connection to prediction theory motivates the following questions: Is it already sufficient to know only parts of the past to determine the future events of $(X_n)_n$? If so, how large a part of the past can be 'forgotten'? In order to make the question more precise, consider the family of exponentials
\[
E(\Lambda):=\mathrm{span}\{z^\lambda\mid z\in \T,\lambda\in \Lambda\}
\]
for $\Lambda\subseteq \N_0$. Furthermore, define
\[
\A:=\{(\Lambda,\mathcal{C})\mid \Lambda\subseteq \N_0, \mathcal{C}\subseteq\Meas, E(\Lambda)\text{ dense in } L^2(\T,\mu)\text{ for all }\mu\in\mathcal{C}\}.
\]
\begin{question}\label{Question}
For which $\Lambda\subseteq\N_0$ and $\mathcal{C}\subseteq \Meas$ is $(\Lambda,\mathcal{C})\in \A$?
\end{question}
\begin{remark}\label{ObviousFacts} We observe the following obvious facts.
\begin{itemize}
    \item[1.)] By Szeg\ho's theorem, for any $\mu\in\Sz$ and $\Lambda\subseteq \N_0$, $E(\Lambda)$ is not dense in $L^2(\T,\mu)$. Thus, for every pair $(\Lambda,\mathcal{C})\in \A$, it follows that $\mathcal{C}\subseteq \Sz^c$.
    \item[2.)] One has the following implications
\begin{align*}
    \Lambda'\subseteq\Lambda, (\Lambda',\mathcal{C})\in \A&\Rightarrow (\Lambda,\mathcal{C})\in \A,\\
    \mathcal{C}'\subseteq \mathcal{C}, (\Lambda,\mathcal{C})\in \A&\Rightarrow (\Lambda,\mathcal{C}')\in \A.
\end{align*}
Thus, the most difficult part about Question~1 is to make $\Lambda$ as small and $\mathcal{C}$ as large as possible.
\item[3.)] It is easy to show that $(\N_0\backslash \Gamma,\Sz^c)\in\A$ for all finite $\Gamma\subseteq \N_0$ (see Corollary~\ref{FiniteGamma}).
\end{itemize}
\end{remark} 
If one studies Question~\ref{Question} for a proper subclass $\mathcal{C}$, instead of considering the maximal class $\Sz^c$,  Olevskii and Ulanovskii \cite{OlevskiiUlanovskii}*{Theorem 1} have obtained a result for the class
\[
    \mathcal{W}:=\Big\{\mu=w \,\dd m\mid w>0\; \dd m\text{-a.e.}, w \text{ bounded}, w \text{ increasing on } (0,2\pi) \Big\}.
\]
\begin{theorem*}[Olevskii, Ulanovskii]
  Let $\Gamma\subseteq \N$ with
  \[
  \sum_{\gamma \in \Gamma} \frac{1}{\sqrt{\gamma}}<\infty.
  \]
 Then $(\N_0\backslash \Gamma,\mathcal{W}\cap \Sz^c)\in\A$.
\end{theorem*}
In this paper, we will add an answer to Question~\ref{Question} for sets $\Lambda$ of the form
\[\Lambda(\mathbf{k},\bm{\ell}):=\bigcup_{j\in\N} \llbracket k_j,k_j+\ell_j\rrbracket,\quad\quad(\mathbf{k}\in\N^\N,\bm{\ell}\in\N_0^\N)\]
where $\mathbf{k}$ is strictly increasing and $\llbracket n,m \rrbracket:=[n,m]\cap \Z$ for $n,m\in\Z, n\leq m$ and where $\mathcal{C}$ is the class of Markoff measures $\MMeas$. This class was introduced by Khrushchev in \cite{KhrushchevClassifictation}. Informally, $\mu$ belongs to $\MMeas$ if the sequence of Verblunsky coefficients $(\alpha_n(\mu))_{n}$ contains a sufficiently dense subsequence that remains bounded away from zero.
We will give a formal definition of Markoff measures and Verblunsky coefficients in Section~2. In Section 4 we will show the following.
\begin{theorem}\label{Existence of infinite sets}
    Let $s>1$ and let $\mathbf{k}:=(k_j)_j\in\N^\N$ be strictly increasing. Set $\lfloor\mathbf{k}^s\rfloor:=(\lfloor k_j^s \rfloor)_j\in\N^\N$. Then, 
    \[
    \Big(\Lambda\big(\mathbf{k},\lfloor \mathbf{k}^s\rfloor\big),\,\MMeas\Big)\in\A.
    \]
    In particular, there exists a subset $\Lambda\subseteq \N$ with lower density $\underline{\dd}\big(\Lambda\big)=0$ such that \linebreak$(\Lambda,\MMeas)\in\A$. 
\end{theorem}
Theorem~\ref{Existence of infinite sets} will follow from the main result of this article, Theorem~\ref{Theorem beta}, which states that for every $\mu\in\Sz^c$ and every strictly increasing $\mathbf{k}\in\N^\N$ there exists $\bm{\ell}\in\N_0^\N$ such that $E(\Lambda(\mathbf{k},\bm{\ell}))$ is dense in $L^2(\T,\mu)$. In Corollary~\ref{MainCorollary} we further show that $\mathbf{k}$ and $\bm{\ell}$ can be chosen such that $\underline{\dd}(E(\Lambda(\mathbf{k},\bm{\ell})))=0$.

It should be noted that \cite{OlevskiiUlanovskii} and this paper explore very different situations. For every $\mu\in\mathcal{W}$ one has $\alpha_n\to 0$ by Rakhmanov's theorem \cite{Rakhmanoc83}, thus  $\MMeas$ and $\mathcal{W}$ are disjoint. Moreover, every measure contained in $\mathcal{W}$ is absolutely continuous with full support. In contrast, if $\mu\in\MMeas$ then $\mathrm{supp}(\mu_{ac})\neq \T$ and also, $\MMeas$ contains measures with $\mu_s\neq 0$ and pure point measures with full support which are important in physics because of the Anderson localization phenomenon (see Lemma~\ref{LemmaMarkoffMeasure} and Remark~\ref{RemarkMarkoffMeasures}). 

The structure of this paper is as follows. In Section 2 we briefly review some notions from OPUC theory. In particular, we state Szeg\ho's theorem (Theorem~\ref{Szego Theorem}) and give a definition and some examples of Markoff measures. In Section 3 we prove the main theorem, Theorem~\ref{Theorem beta}. In Section 4 we prove Theorem~\ref{Existence of infinite sets} and show other applications of Theorem~\ref{Theorem beta}.
\subsection*{Notation} 
\begin{itemize}
    \item We use the conventions $\N:=\{n\in\Z\mid n\geq 1\}$ and $\N_0:=\N\cup\{0\}$.
    \item Interval of integers: For $n,m\in \Z$  with $n\leq m$ define
    \[
    \llbracket n,m \rrbracket:=[n,m]\cap \Z.
    \]
    \item Upper and lower density: For $\Lambda\subseteq \N$ define the lower and upper density by
    \[
    \underline{\dd}(\Lambda):=\liminf_{N\to\infty} \frac{|\Lambda\cap \llbracket 1,N\rrbracket|}{N},\quad \overline{\dd}(\Lambda):=\limsup_{N\to\infty}\frac{|\Lambda\cap \llbracket 1,N\rrbracket|}{N}
    \]
    respectively.
    \item $\Meas$ is the set of all non-degenerate probability measures on $\T$.
    \item Closure of a subset: For $\mu\in\Meas$ and a subset $M\subseteq L^2(\T,\mu)$ we write 
    \[
    \overline{M}^{L^2(\T,\mu)}
    \]
    for the closure of $M$ in $L^2(\T,\mu)$.
    \item $\Sz$ is the set of all measures $\dd\mu=w\,\dd m+\dd\mu_s\in\Meas$ for which
    \[
    \int_0^{2\pi}\log(w(\theta))\,\dd\theta>-\infty.
    \]
    \item $\MMeas$ denotes the set of all Markoff measures (see definition on page \pageref{MarkoffMeasureDefinition}).
    \item We already defined earlier
    \begin{align*}
        E(\Lambda)&=\mathrm{span}\{z^\lambda\mid z\in \T,\lambda\in \Lambda\},\\
        \hspace{2cm}\A&=\{(\Lambda,\mathcal{C})\mid \Lambda\subseteq \N_0, \mathcal{C}\subseteq\Meas, E(\Lambda)\text{ dense in } L^2(\T,\mu)\text{ for all }\mu\in\mathcal{C}\},\\
        \Lambda(\mathbf{k},\bm{\ell})&=\bigcup_{j\in\N} \llbracket k_j,k_j+\ell_j\rrbracket,\quad\quad(\mathbf{k}\in\N^\N,\bm{\ell}\in\N_0^\N)
    \end{align*}
    where $\mathbf{k}$ is strictly increasing.
    \item Let $\PP$ be the space of all polynomials. For $n\in\N_0$, let $\PP_{\leq n}$ be the space of all polynomials with degree less than or equal to $n$.
    \item We will also need the following definitions later
    \begin{align*}
        \beta_\mu(k,n)&:=\min_{\pi\in \PP_{\leq n}}\|z^{-k}-\pi(z)\|_{L^2(\T,\mu)},\hspace{-1cm}&&(k,n\in\N_0,\mu\in\Meas),\\
        \beta_\mu(k,\infty)&:=\lim_{n\to\infty}\beta_\mu(k,n)\\
        &\;=\inf_{\pi\in \PP}\|z^{-k}-\pi(z)\|_{L^2(\T,\mu)},\hspace{-1cm}&&(k\in\N_0,\mu\in\Meas).
    \end{align*}
\end{itemize}

\subsection*{Acknowledgment.}
    I am very grateful to my advisor, Professor Sergey Denisov, for suggesting the problem formulated as Question 1, and for his countless helpful discussions and comments.
\section{Facts from OPUC}
\begin{theorem}[Szeg\ho]\label{Szego Theorem}
   Let $\mu\in\Meas$ with $\dd \mu=w\,\dd m+\dd\mu_s$.  Then
    \[
    \exp\Big(\int_0^{2\pi} \log w\,\dd m\Big)=\beta_\mu(1,\infty).
    \]
\end{theorem}
This appears first in \cite{SzegoGrenander58}*{Chapter 3.1}, see also \cite{SimonOPUC1}*{Chapter 2}. We have the following corollary \cite{SimonOPUC1}*{Theorem~1.5.7}.
\begin{corollary}\label{Szego Corollary}
    Let $\mu\in\Meas$ with $\dd \mu=w\,\dd m+\dd\mu_s$. Then $\mu\in\Sz^c$ if and only if the polynomials are dense in $L^2(\T,\mu)$, i.e. $\beta_\mu(k,\infty)=0$ for all $k\in\N$.
\end{corollary}
Let $\mu\in\Meas$ with $\dd \mu=w\,\dd m+\dd\mu_s$. Because of the non-degeneracy of $\mu$, the monomials $z^n$ are linearly independent in $L^2(\T,\mu)$. Thus, for every $n\in\N_0$ there is a uniquely determined polynomial
\[
\Phi_{n}(\mu;z)=\sum_{i=0}^n b_{n,i}(\mu)z^i\quad (z\in\T)
\]
such that $b_{n,i}\in\C$ for every $i\in\llbracket 0,n\rrbracket$, $b_{n,n}=1$ and, for every $k\in\llbracket 0,n-1\rrbracket$, 
\[
\big\langle \Phi_{n}(\mu;z),\Phi_{k}(\mu;z)\big\rangle_{L^2(\T,\mu)}=0.
\]
The polynomials $\Phi_{n}(\mu;z)$ are called \textit{monic orthogonal polynomials of} $\mu$. The monic orthogonal polynomials satisfy a recursion relation, the \textit{Szeg\H{o} recurrence},
\[
\Phi_{n+1}(\mu;z)=z\Phi_{n}(\mu;z)-\overline{\alpha_n(\mu)}\Phi^{*}_n(\mu;z)
\]
for every $n\in\N_0$, where, for $z\in\T$,
\[
\Phi^{*}_n(\mu;z)=z^n\overline{\Phi_{n}(\mu;z)}=\sum_{i=0}^n \overline{b_{n,n-i}(\mu)}z^i.
\]
For $n\in\N_0$ the recursion coefficient $\alpha_n(\mu)$ is called the $n$-th \textit{Verblunsky coefficient} of $\mu$. One has $\alpha_n(\mu)\in\mathbb{D}$ for every $n\in\N_0$. In fact, the mapping
\[
\Meas \to \mathbb{D}^{\N_0},\quad \mu \mapsto (\alpha_n(\mu))_n
\]
is a bijection.
This is known as \textit{Verblunksy's theorem} \cite{SimonOPUC1}*{Theorem~1.7.11}. One can also express the $L^2$-norms of the monic orthogonal polynomials and the integral of $\log(w)$ by the Verblunsky coefficients \cite{SimonOPUC1}*{Theorems~1.5.2,~2.3.1}
\begin{align}\label{VerblunskyIdentity}
     \|\Phi_{n+1}(\mu;z)\|^2_{L^2(\T,\mu)}&=\prod_{i=0}^{n} (1-|\alpha_i(\mu)|^2),\\\nonumber
     \exp\Big(\int_0^{2\pi}\log w\;\mathrm{d}m\Big)&=\lim_{n \to\infty} \|\Phi_{n}(\mu;z)\|^2_{L^2(\T,\mu)}=\prod_{i=0}^\infty (1-|\alpha_i(\mu)|^2).\label{VerblunskyIdentity}
\end{align}
The second equation proves the following corollary.
\begin{corollary}\label{Verblunkyl2Condition}
    Let $\mu\in\Meas$. Then $\mu\in\Sz$ if and only if $(\alpha_n(\mu))_n\in\ell^2$.
\end{corollary}
In the following, we will write $\Phi_n, b_{n,i}$ and $\alpha_n$ instead of $\Phi_n(\mu;z), b_{n,i}(\mu)$ and $\alpha_n(\mu)$ whenever the measure $\mu$ is clear from the context.
\subsection*{Markoff measures}\label{MarkoffMeasureDefinition}
Let $\mu\in\Meas$.  For $\ell\in\N_0$ and $\varepsilon\in(0,1)$ define the class $\mathrm{Mar}_{\varepsilon,\ell}(\T)\subseteq \Meas$ via
\begin{equation}\label{markoffmeasure}
\mu\in\mathrm{Mar}_{\varepsilon,\ell}(\T)\Leftrightarrow \inf_{n\in\N_0}\max_{n\leq j\leq n+\ell} |\alpha_j|\geq \varepsilon.
\end{equation}
The class $\MMeas$ of \textit{Markoff measures}, which was introduced by Khrushchev in \cite{KhrushchevClassifictation}, is defined as follows
\[
\MMeas:=\bigcup_{\ell\in\N_0}\bigcup_{\varepsilon>0} \mathrm{Mar}_{\varepsilon,\ell}(\T).
\]
\begin{remark}
    This definition is in fact a characterization of the original definition \cite{KhrushchevClassifictation}*{Theorem~1.8}. Khrushchev defines $\MMeas$ as every $\mu\in\Meas$ such that $\dd m$ is not contained in the derived set of
    \[
    \{|\varphi_n|^2\dd \mu\mid n\in\N_0 \}
    \]
    where $\varphi_n$ is the $n$-th orthogonal polynomial in $L^2(\T,\mu)$ obtained by using the Gram-Schmidt algorithm on the monomials $z^n$ with $n\in\N_0$.
\end{remark}
Let us collect a few properties and examples of Markoff measures. 
\begin{lemma}[Properties of Markoff measures]\label{LemmaMarkoffMeasure}
Let $\mu\in\Meas$.
\begin{itemize}
\setlength{\itemsep}{1ex}
    \item[a)]  $\MMeas\subseteq \Sz^c$.
    \item[b)] $\mu\in\mathrm{Mar}_{\varepsilon,\ell}(\T)\Rightarrow \|\Phi_n\|^2_{L^2(\T,\mu)}\leq (1-\varepsilon^2)^{\frac{n-1}{\ell+1}-1}$.
    \item[c)] $\mu\in\MMeas\Rightarrow\mathrm{supp}(\mu_{ac})\neq \T$. 
    \item[d)] $\mathrm{supp}(\mu)\neq \T\Rightarrow \mu\in\MMeas$.
\end{itemize}
The converse implications in c) and d) are not true.
\end{lemma}
\begin{proof}
a) Let $\mu\in\MMeas$. Then, $(\alpha_n)_n$ does not converge to zero by definition of $\MMeas$. In particular, $(\alpha_n)_n\notin \ell^2$. Thus, by Corollary \ref{Verblunkyl2Condition}, $\mu\in \Sz^c$.

b)  Let $\mu\in\mathrm{Mar}_{\varepsilon,\ell}(\T)$. From \eqref{VerblunskyIdentity} it follows that
\[
\|\Phi_n\|^2_{L^2(\T,\mu)}=\prod_{i=0}^{n-1} (1-|\alpha_i|^2)\leq (1-\varepsilon^2)^{\lfloor\frac{n-1}{\ell+1}\rfloor}\leq(1-\varepsilon^2)^{\frac{n-1}{\ell+1}-1}.
\]

c) By Rakhmanov's theorem \cite{Rakhmanoc83}*{§3}, \cite{SimonOPUC2}*{Corollary~9.1.11}, for any $\mu\in\Meas$ with $\mathrm{supp}(\mu_{ac})=\T$ one has 
\[
\lim_{n\to\infty}\alpha_n(\mu)=0.
\]
Let $\mu\in\MMeas$. Then $(\alpha_n(\mu))_n$ does not converge to 0. Thus, by Rakhmanov's theorem, $\mathrm{supp}(\mu_{ac})\neq \T$. To see that the converse implication of c) doesn't hold true, consider a measure $\mu\in\Meas$ such that for all $n\in\N_0$
\[
\alpha_n(\mu)=
\begin{cases}
    \frac{1}{\sqrt{k}}, & n=k!, k\in\N_0\\
    0, &\text{else}.
\end{cases}
\]
By Corollary \ref{Verblunkyl2Condition} $\mu\in\Sz^c$. However, $\mu\notin\MMeas$ since ${a_n\to 0}$. Furthermore, $\mu$ is purely singular continuous (see \cite{SimonOPUC2}*{Theorem~12.5.2}). In particular, $\mu_{ac}=0$.

d) For a proof see \cite{KhrushchevClassifictation}*{Corollary~1.9}. We briefly present a counterexample to the converse implication due to Zhedanov \cite{Zhedanov20}. For $q\in\T$ which is not a root of unity and $0<p<1$ set
\[
\mu:=(1-p)\sum_{n=0}^\infty p^n\delta_{q^n}
\]
where $\delta_w$ is the Dirac measure with $\mathrm{supp}(\delta_w)=\{w\}$. Clearly, $\mu\in\Meas$. There is also a simple formula for $|\alpha_n|^2$ in terms of $p$ and $q$ (see \cite{Zhedanov20}*{p.5}). For all $n\in\N_0$
\[
|\alpha_n|^2=\frac{(1-p)^2}{1+p^2-2p\,\mathrm{Re}(q^{n+1})}.
\]
Thus, 
\[
\frac{1-p}{1+p}<|\alpha_n|<1
\]
for all $n\in\N_0$ which implies that $\mu\in\MMeas$. Furthermore, since $q$ is not a root of unity, 
\[
\mathrm{supp}(\mu)=\overline{\{q^n\mid n\in\N_0\}}=\T.\qedhere
\]
\end{proof}
\begin{remark}\label{RemarkMarkoffMeasures}
    In some sense, measures which are pure point with full support are generic examples. Namely, for any rotation invariant probability measure $\beta_0$ on $\D$ with $\beta_0\neq \delta_0$ and
    \[
    \int_{\D} \log(1-|\omega|)\,\dd \beta_0(\omega)>-\infty
    \]
        one has that almost every sequence $(\omega_n)_{n}\in \D^{\N_0}$ with respect to the product measure
        \[
        \beta:=\bigotimes_{i=0}^\infty \beta_0
        \]
        generates a measure $\mu$, by choosing $\mu$ with $\alpha_n(\mu)=\omega_n$ for every $n\in\N_0$, that is pure point and has full support \cite{SimonOPUC2}*{Theorem~12.6.2}. One can also view the Verblunsky coefficients $\alpha_n$ as being values of an i.i.d. process $(\omega_n)_n$ with common distribution $\beta_0$. Such measures are important in physics because, going back to the work of Anderson \cite{Anderson58}, certain random Hamiltonians on lattices produce dense point spectra - a phenomenon called \textit{Anderson localization}. Mathematical references include \cite{Carmona90} and \cite{Pastur92}.
\end{remark}
\section{Main Theorem}
Recall that for $k,n\in\N_0$
\[
\beta_\mu(k,n)=\min_{\pi\in \PP_{\leq n}}\|z^{-k}-\pi(z)\|_{L^2(\T,\mu)}.
\]
\begin{definition}
    Let $\mu\in\Meas$. A strictly increasing function $f:\N\to\N$ is called $\beta$\textit{-approximating} for $\mu$ if 
    \begin{equation*}
        \lim_{k\to\infty} \beta_\mu(k,f(k))=0.
    \end{equation*}
\end{definition}
\begin{corollary}\label{CorollaryBetaApproximating}
    Let $\mu\in\Meas$. Then there exists a $\beta$-approximating $f$ if and only if $\mu\in\Sz^c$.
\end{corollary}
\begin{proof}
    Let $\mu\in\Sz^c$. Then by Corollary~\ref{Szego Corollary} 
    \[
        \lim_{n\to\infty} \beta_\mu(k,n)=\beta_\mu(k,\infty)=0
    \]
    for every $k\in\N$.
    Thus there exists an $f:\N\to\N$ that is $\beta$-approximating for $\mu$.

    For the converse implication, let $f:\N\to\N$ be $\beta$-approximating for $\mu$. Since $(\beta_\mu(k,n))_n$ is a decreasing sequence for every $k\in\N$ and $f$ is  increasing
    \[
    \beta_\mu(k,\infty)=\inf_{k\in\N} \beta_\mu(k,f(k))=\lim_{k\to\infty} \beta_\mu(k,f(k))=0.
    \]
    Thus, by Corollary~\ref{Szego Corollary}, $\mu\in \Sz^c$.
\end{proof}
 Recall also, that for two sequences $\bm{\ell}:=(\ell_j)_j\in\N_0^\N, \mathbf{k}:=(k_j)_j\in\N^\N$, where $\mathbf{k}$ is strictly increasing, 
    \[
    \Lambda(\mathbf{k},\bm{\ell})=\bigcup_{j\in\N} \llbracket k_j,k_j+\ell_j\rrbracket.
    \]
    Note that the definition of $\Lambda(\mathbf{k},\bm{\ell})$ does not require the intervals $\llbracket k_j,k_j+\ell_j\rrbracket$ to be disjoint.
    
    Now we are ready to state the main theorem.
\begin{theorem}\label{Theorem beta}
Let $\mu\in\Sz^c$ and $f$ be any $\beta$-approximating function for $\mu$. Let $\mathbf{k}:=(k_j)_j\in\N^\N$ be strictly increasing and $\bm{\ell}:=(\ell_j)_j\in\N_0^\N$ such that 
    \[
    \lim_{j\to\infty} \ell_j-f(k_j)=\infty.
    \]  
    Then $E(\Lambda(\mathbf{k},\bm{\ell}))$ is dense in $L^2(\T,\mu)$.
\end{theorem}
Let us first prove the following two corollaries.
\begin{corollary}\label{MainCorollary}
    Let $\mu\in \Sz^c$. Then there exists a set $\Lambda\subseteq \N$ with $\underline{\dd}(\Lambda)=0$ such that $E(\Lambda)$ is dense in $L^2(\T,\mu)$.
\end{corollary}
\begin{proof}
    By Corollary \ref{CorollaryBetaApproximating} there exists $f:\N\to\N$ which is $\beta$-approximating for $\mu$. Set $k_1:=1$ and, for $j\in\N$, define $k_{j+1}$ recursively such that 
   \[
   (k_j+f(k_j)+j)j+1\leq k_{j+1}.
   \]
   Put $\ell_j:=f(k_j)+j$ for every $j\in\N$, $\bm{\ell}:=(\ell_j)_j$ and $\mathbf{k}:=(k_j)_j$. By Theorem~\ref{Theorem beta}, $E(\Lambda(\mathbf{k},\bm{\ell}))$ is dense in $L^2(\T,\mu)$.  Furthermore, we get the following estimate for~$\Lambda(\mathbf{k},\bm{\ell})$. 
\begin{align*}
    \underline{\dd}(\Lambda(\mathbf{k},\bm{\ell}))&\leq \liminf_{j\to\infty} \frac{\big|\Lambda(\mathbf{k},\bm{\ell})\cap \llbracket 1,k_{j+1}-1\rrbracket\,\big|}{k_{j+1}-1}\\[1em]
    &\leq \liminf_{j\to\infty} \frac{k_j+\ell_j}{k_{j+1}-1}\leq \liminf_{j\to\infty}\frac{1}{j}=0.\qedhere
\end{align*}
\end{proof}
\begin{corollary}\label{FiniteGamma}
    $(\N_0\backslash \Gamma,\Sz^c)\in\A$ for every finite $\Gamma\subseteq \N_0$. 
\end{corollary}
\begin{proof}
For $\mu\in\Sz^c$ choose $f$ which is $\beta$-approximating for $\mu$. Set
    \[
    k_1:=\max_{\gamma\in\Gamma}\gamma+1,\quad k_{j+1}:=k_j+f(k_j)+j
    \]
    for all $j\in\N$. Set furthermore $\ell_j:=f(k_j)+j$ for all $j\in\N$. Then 
    \[
    \Lambda(\mathbf{k},\bm{\ell})=\N_0\backslash \llbracket 0,\max_{\gamma\in\Gamma}\gamma\rrbracket\subseteq \N_0\backslash\Gamma.
    \]
    By Theorem~\ref{Theorem beta}, $E(\Lambda(\mathbf{k},\bm{\ell}))$ is dense in $L^2(\T,\mu)$. Thus, also $E(\N_0\backslash \Gamma)$ is dense in $L^2(\T,\mu)$.     
\end{proof}
\begin{remark}
    Corollary \ref{FiniteGamma} is an extension of Szeg\ho's classical theorem. Szeg\ho's theorem is the case $\Gamma=\emptyset$.
    \end{remark}
Before we can prove the main theorem, we first prove two lemmas about approximating negative powers of $z$ with polynomials in $L^2(\T,\mu)$. Throughout, we will use the convention $\sum_{j\in J} a_j:=0$ if $J=\emptyset$.
\begin{lemma}\label{LemmaDistance}
    Let $\mu\in \Meas$. For all $k,n\in \N$ we have the estimate
    \begin{align}\label{Eq:EstimateBeta}
        \beta_\mu(k,n)&\leq \|\Phi_{n+1}\|_{L^2(\T,\mu)} \Bigg(1+\sum_{i=1}^{k-1}\;\:\sum_{\substack{ j_1,...,j_i\in \llbracket 1,n+1\rrbracket\\[0.5ex]j_1+...+j_i\leq k-1 }} \:|b_{n+1,n+1-j_1}|\cdot...\cdot |b_{n+1,n+1-j_i}| \Bigg).
    \end{align}
\end{lemma}
\begin{lemma}\label{LemmaDistance2}
    Let $\mu\in\Meas$. Then, for all $k,n\in\N$
    \[
    \beta_\mu(k,n)\leq\|\Phi_{n+1}\|_{L^2(\T,\mu)} (2n+2)^{k-1}.
    \]
\end{lemma}
\begin{proof}[Proof of Lemma \ref{LemmaDistance}]
     The proof is by induction. For $k=1$ and $n\in \N$ we show
     \[
     \beta_\mu(1,n)=\|\Phi_{n+1}\|_{L^2(\T,\mu)}.
     \]
     This is a standard result in OPUC theory but the proof is included here for completeness.
    
    For a subspace $V\subseteq L^2(\T,\mu)$ denote by $P_V$ the orthogonal projection onto $V$. 
    \begin{align*}
        \beta_\mu(1,n)&=\min\big\{\|z^{-1}-\pi(z)\|_{L^2(\T,\mu)}\mid \pi\in \PP_{\leq n} \big\}\\
        &=\min\big\{\|1-\pi(z)\|_{L^2(\T,\mu)}\mid \pi\in z\cdot\PP_{\leq n}\big\}\\
        &=\|P_{(z\cdot\PP_{\leq n})^\perp}(1)\|_{L^2(\T,\mu)}.
    \end{align*}
    Note that $\Phi^{*}_{n+1}(0)=b_{n+1,n+1}=1$ and thus
    \[
    \Phi^{*}_{n+1}(z)-1=\sum_{i=1}^{n+1} \overline{b_{n+1,n+1-i}}z^i\in z\cdot\PP_{\leq n}.
    \]
    Furthermore, for all $\pi\in\PP_{\leq n}$,
    \begin{align*}
        \langle z\pi,\Phi^{*}_{n+1}\rangle_{L^2(\T,\mu)}=\langle z\pi,z^{n+1}\overline{\Phi_{n+1}}\rangle_{L^2(\T,\mu)}=\langle \Phi_{n+1}, z^n\overline{\pi}\rangle_{L^2(\T,\mu)}=0
    \end{align*}
    since $\Phi_{n+1}\perp \PP_{\leq n}$. Thus $P_{(z\cdot\PP_{\leq n})^\perp}(1)=\Phi^{*}_{n+1}$ and therefore
    \[
     \beta_\mu(1,n)=\|\Phi^{*}_{n+1}\|_{L^2(\T,\mu)}=\|\Phi_{n+1}\|_{L^2(\T,\mu)}.
    \]
    Now, we make an induction in $k$. Assume that for all $n\in\N$ one can estimate $\beta_\mu(1,n),\ldots,\beta_\mu(k,n)$ from above by the right-hand side in \eqref{Eq:EstimateBeta}. By applying the triangle inequality one gets for all $n\in\N$
    \begin{align}\label{EQ:TriangleInequalityBetaEstimate}
        \beta_\mu(k+1,n)&\leq \|z^{-(k+1)}+\sum_{m=1}^{n+1} \overline{b_{n+1,n+1-m}}\,P_{
        (\PP_{\leq n})}(z^{m-(k+1)})\|_{L^2(\T,\mu)}\nonumber\\[1ex]
        &\leq  \|z^{-(k+1)}+\sum_{m=1}^{n+1} \overline{b_{n+1,n+1-m}}\,z^{m-(k+1)}\|_{L^2(\T,\mu)}\nonumber\\[1ex]
        &\hspace{0.9cm}+\sum_{m=1}^{\min
        (k,n+1)} |b_{n+1,n+1-m}|\cdot \|z^{m-(k+1)}-P_{
        (\PP_{\leq n})}(z^{m-(k+1)})\|_{L^2(\T,\mu)}\nonumber\\[1ex]
        &=\|z^{-(k+1)}+\sum_{m=1}^{n+1} \overline{b_{n+1,n+1-m}}\,z^{m-(k+1)}\|_{L^2(\T,\mu)}\nonumber\\[1ex]
        &\hspace{1.9cm}+\sum_{m=1}^{\min(k,n+1)} |b_{n+1,n+1-m}|\, \beta_\mu(k+1-m,n).
    \end{align}
    The first summand on the right-hand side of \eqref{EQ:TriangleInequalityBetaEstimate} equals
    \begin{equation}\label{Eq:FirstSummandBetaEstimate}
    \|z^{-(k+1)}+z^{-(k+1)}(\Phi^{*}_{n+1}(z)-1)\|_{L^2(\T,\mu)}=\|\Phi^{*}_{n+1}(z)\|_{L^2(\T,\mu)}=\|\Phi_{n+1}\|_{L^2(\T,\mu)} .
    \end{equation}
    By the induction hypothesis, each $\beta_\mu(k+1-m,n)$ in the second summand on the right-hand side of \eqref{EQ:TriangleInequalityBetaEstimate} can be bounded from above by the following expression 
    \begin{align}\label{Eq:BetaEstimateCalculation}
    \|\Phi_{n+1}\|_{L^2(\T,\mu)}\Big(1+\sum_{i=1}^{k-1}\sum_{\substack{  j_1,\ldots,j_i\in\llbracket 1,n+1\rrbracket,\\[0.5ex]j_1+...+j_i\leq k-m }}\ |b_{n+1,n+1-j_1}|\cdot...\cdot |b_{n+1,n+1-j_i}|\Big).
    \end{align}
    Note that when applying the induction hypothesis to $\beta_\mu(k+1-m,n)$ we would a priori get a sum in $i$ that goes from 1 to $k-m$. The reason why the sum in \eqref{Eq:BetaEstimateCalculation} goes from 1 to $k-1$ instead is that for $i>k-m$ the index set 
    \[
    \{j_1,\ldots,j_i\in\llbracket 1,n+1\rrbracket\mid j_1+...+j_i\leq k-m\}
    \]
    is empty and thus the sum over this index set equals zero. 
    By \eqref{Eq:BetaEstimateCalculation} and \eqref{EQ:TriangleInequalityBetaEstimate}, the second summand on the right-hand side of \eqref{EQ:TriangleInequalityBetaEstimate} is bounded from above by
    \begin{align}\label{Eq:BetaEstimateCalculation3}
       &\|\Phi_{n+1}\|_{L^2(\T,\mu)} \Bigg(\sum_{m=1}^{\min(k,n+1)} |b_{n+1,n+1-m}|\nonumber\\[1ex]
    &\hspace{1cm}+\sum_{m=1}^{\min(k,n+1)}\sum_{i=1}^{k-1} \sum_{\substack{  j_1,\ldots,j_i\in\llbracket 1,n+1\rrbracket,\\[0.5ex]j_1+...+j_i\leq k-m }} |b_{n+1,n+1-m}|\cdot|b_{n+1,n+1-j_1}|\cdot...\cdot |b_{n+1,n+1-j_i}|\Bigg).
    \end{align}
    By rearranging the summands in \eqref{Eq:BetaEstimateCalculation3} we get
    \begin{align}\label{Eq:BetaEstimateCalculation2}
    &\|\Phi_{n+1}\|_{L^2(\T,\mu)} \Bigg(\sum_{m=1}^{\min(k,n+1)} |b_{n+1,n+1-m}|\nonumber\\[1ex]
    &\hspace{1cm}+\sum_{i=1}^{k-1}\sum_{m=1}^{\min(k,n+1)} \sum_{\substack{  j_1,\ldots,j_i\in\llbracket 1,n+1\rrbracket,\\[0.5ex]j_1+...+j_i\leq k-m }} |b_{n+1,n+1-m}|\cdot|b_{n+1,n+1-j_1}|\cdot...\cdot |b_{n+1,n+1-j_i}|\Bigg)\nonumber\\[1ex]
    &=\|\Phi_{n+1}\|_{L^2(\T,\mu)} \Bigg(\sum_{m=1}^{\min(k,n+1)} |b_{n+1,n+1-m}|\nonumber\\[1ex]
    &\hspace{3cm}+\sum_{i=2}^{k} \sum_{\substack{  j_1,\ldots,j_i\in\llbracket 1,n+1\rrbracket,\\[0.5ex]j_1+...+j_i\leq k }} |b_{n+1,n+1-j_1}|\cdot...\cdot |b_{n+1,n+1-j_i}|\Bigg).
    \end{align}
    In \eqref{Eq:BetaEstimateCalculation2}, the sum in $m$ is equal to the sum in $j_1,\ldots,j_i$ in the case $i=1$. Thus, the right-hand side in \eqref{Eq:BetaEstimateCalculation2} equals
    \begin{equation*}
        \|\Phi_{n+1}\|_{L^2(\T,\mu)} \sum_{i=1}^{k}\sum_{\substack{  j_1,\ldots,j_i\in\llbracket 1,n+1\rrbracket,\\[0.5ex]j_1+...+j_i\leq k }} |b_{n+1,n+1-j_1}|\cdot...\cdot |b_{n+1,n+1-j_i}|.
    \end{equation*}
    Now we have an upper bound for the second summand in the right-hand side of \eqref{EQ:TriangleInequalityBetaEstimate}. In \eqref{Eq:FirstSummandBetaEstimate} we asserted that the first summand in the right-hand side of \eqref{EQ:TriangleInequalityBetaEstimate} equals $\|\Phi_{n+1}\|_{L^2(\T,\mu)}$. Putting both together we obtain the desired upper bound 
    \begin{align*}
         &\beta_\mu(k+1,n)\\
         &\leq \|\Phi_{n+1}\|_{L^2(\T,\mu)} \Bigg(1+\sum_{i=1}^{k}\:\; \sum_{\substack{  j_1,\ldots,j_i\in\llbracket 1,n+1\rrbracket,\\[0.5ex]j_1+...+j_i\leq k }}\: |b_{n+1,n+1-j_1}|\cdot...\cdot |b_{n+1,n+1-j_i}|\Bigg).\qedhere
    \end{align*}
\end{proof}
\begin{proof}[Proof of Lemma \ref{LemmaDistance2}]
    Let $w_1,...,w_n$ be the zeros of $\Phi_n$. Since the zeros of monic orthogonal polynomials on the unit circle are contained in the unit disk \cite{SimonOPUC1}*{Theorem~1.7.1} we have 
    \begin{equation}\label{Eq:BetaBinomialCoefficient}
    |b_{n,k}|\leq \sum_{\substack{J\subseteq \llbracket 1,n\rrbracket,\: |J|=n-k}}\; \;\prod_{j\in J} |w_j|\leq \binom{n}{k}.
    \end{equation}
    Together with Lemma \ref{LemmaDistance} we get a new estimate for $\beta_\mu(k,n)$ by using  \eqref{Eq:BetaBinomialCoefficient} to bound each coefficient $b_{n+1,n+1-j_i}$ in the right-hand side of \eqref{Eq:EstimateBeta} 
    \begin{align*}
         \beta_\mu(k,n)&\leq \|\Phi_{n+1}\|_{L^2(\T,\mu)}\Bigg( 1+\sum_{i=1}^{k-1}\;\:\sum_{\substack{ j_1,...,j_i\in \llbracket1,n
         +1\rrbracket\\[0.5ex]j_1+...+j_i\leq k-1 }} \:\binom{n+1}{j_1}\cdot...\cdot \binom{n+1}{j_i}\Bigg).
    \end{align*}
    For each of the binomial coefficients, we now use the bound $\binom{n}{k}\leq n^k$. In general, this bound is very wasteful, but in this case, one always gets the summand $
    (n+1)^{k-1}$ from the tuple $(j_1,\ldots,j_{k-1})$ with entries $j_i=1$ for each $i\in\llbracket 1,k-1\rrbracket$. Hence,
    \begin{align}\label{Eq:LemmaBetasEstimate}
         \beta_\mu(k,n)&\leq \|\Phi_{n+1}\|_{L^2(\T,\mu)} \Bigg( 1+ \sum_{i=1}^{k-1}\;\: \sum_{\substack{ j_1,...,j_i\in \llbracket1,n+1\rrbracket\\[0.5ex]j_1+...+j_i\leq k-1 }} \:(n+1)^{j_1}\cdot...\cdot (n+1)^{j_i}\Bigg)\nonumber\\
         &\leq \|\Phi_{n+1}\|_{L^2(\T,\mu)}  \Bigg( 1+\sum_{i=1}^{k-1}\;\:\sum_{\substack{ j_1,...,j_i\in \llbracket1,n+1\rrbracket\\[0.5ex]j_1+...+j_i\leq k-1 }}(n+1)^{k-1}\Bigg)\nonumber\\
         &= \|\Phi_{n+1}\|_{L^2(\T,\mu)}  \Bigg( 1+(n+1)^{k-1}\sum_{i=1}^{k-1}\;\:\sum_{\ell=i}^{k-1}\;\:\sum_{\substack{ j_1,...,j_i\in \llbracket1,\ell\rrbracket\\[0.5ex]j_1+...+j_i=\ell }}1\Bigg).
    \end{align}
   The last sum in \eqref{Eq:LemmaBetasEstimate} now counts the number of tuples $(j_1,\ldots,j_i)$ with entries in $\N$ such that $\sum_{s=1}^i j_s=\ell$. This number is $\binom{\ell-1}{i-1}$ since one has to count how many ways there are to draw $i-1$ separating lines between $\ell$ 1's. Thus, from \eqref{Eq:LemmaBetasEstimate} we get the following estimate for $\beta_\mu(k,n)$
   \begin{equation}\label{Eq:LemmaBetaEstimateNo2}
   \beta_\mu(k,n)\leq\|\Phi_{n+1}\|_{L^2(\T,\mu)}\, \Bigg( 1+(n+1)^{k-1}\sum_{i=1}^{k-1}\;\:\sum_{\ell=i}^{k-1}\;\:\binom{\ell-1}{i-1}\Bigg).
   \end{equation}
   To simplify the right-hand side, we use the following combinatorial identity
   \begin{equation}\label{Eq:CombinatoricIdentity}
       \sum_{\ell=i}^{k-1}\;\:\binom{\ell-1}{i-1}=\binom{k-1}{i}
   \end{equation}
   for $1\leq i\leq k-1$. The identity follows by induction in $k$ from Pascal's rule for binomial coefficients. Indeed, if we assume \eqref{Eq:CombinatoricIdentity} holds true for some $k\geq 2$, then
   \[
   \sum_{\ell=i}^{k}\;\:\binom{\ell-1}{i-1}=\sum_{\ell=i}^{k-1}\;\:\binom{\ell-1}{i-1}+\binom{k-1}{i-1}=\binom{k-1}{i}+\binom{k-1}{i-1}=\binom{k}{i}
   \]
   where the last equation is Pascal's rule.
   Combining \eqref{Eq:CombinatoricIdentity} and \eqref{Eq:LemmaBetaEstimateNo2} yields
   \begin{align*}
       \beta_\mu(k,n)&\leq\|\Phi_{n+1}\|_{L^2(\T,\mu)} \, \Bigg( 1+(n+1)^{k-1}\sum_{i=1}^{k-1}\;\binom{k-1}{i}\Bigg)\\
       &=\|\Phi_{n+1}\|_{L^2(\T,\mu)}\,\big(1+(n+1)^{k-1}(2^{k-1}-1)\big)\\
       &\leq \|\Phi_{n+1}\|_{L^2(\T,\mu)}\,(2n+2)^{k-1}
   \end{align*}
   which is the estimate we wanted to show.
\end{proof}
\begin{proof}[Proof of Theorem~\ref{Theorem beta}]
   Without loss of generality
    \begin{equation}\label{ConditionSequence}
        k_j+\ell_j<k_{j+1}
    \end{equation}
    for every $j\in\N$. Otherwise one passes to subsequences $\mathbf{k}':=(k_{n_j})_j$ and $\bm{\ell}':=(\ell_{n_j})_j$  that fulfill condition \eqref{ConditionSequence}. Since
    $\Lambda(\mathbf{k}',\bm{\ell}')\subseteq \Lambda(\mathbf{k},\bm{\ell})$,
     the density of $E(\Lambda(\mathbf{k}',\bm{\ell}'))$ implies the density of $E(\Lambda(\mathbf{k},\bm{\ell}))$ in $L^2(\T,\mu)$. Thus, we can assume \eqref{ConditionSequence}.
     
     For $k\in \N$ let $\pi_k\in\PP_{\leq f(k)}$ such that 
    \[
        \|z^{-k}-\pi_k\|_{L^2(\T,\mu)}=\min_{\pi\in \PP_{\leq f(k)}}\|z^{-k}-\pi(z)\|_{L^2(\T,\mu)}=\beta_\mu(k,f(k)).
    \]
By Theorem~\ref{Szego Theorem} it suffices to show that 
    \[ \overline{E(\N_0)}^{L^2(\T,\mu)}=\overline{E(\Lambda(\mathbf{k},\bm{\ell}))}^{L^2(\T,\mu)}.
   \]
   To this end, let $k\in\N_0$. We need to show that $z^{k}\in\overline{E(\Lambda(\mathbf{k},\bm{\ell}))}^{L^2(\T,\mu)}$.
   
   Let $J\in\N$ such that
   \begin{equation*}
   k\leq \ell_j-f(k_j)
   \end{equation*} 
   for every $j\geq J$.
   Then, for every $j\geq J$,
   \[
   k_j\leq k+k_j+\mathrm{deg}(\pi_{k_j})\leq k+k_j+f(k_j)\leq k_j+\ell_j.
   \]
It follows that
   \begin{align*}
       z^{k+k_j}\pi_{k_j}\in \mathrm{span}\{z^m\mid k_j\leq m\leq k_j+\ell_j\}\subseteq E(\Lambda(\mathbf{k},\bm{\ell}))
   \end{align*}
   for every $j\geq J$. Furthermore, for every $j\in\N$,
    \[
   \|z^{k}-z^{k+k_j}\pi_{k_j}\|_{L^2(\T,\mu)}= \|z^{-k_j}-\pi_{k_j}\|_{L^2(\T,\mu)}=\beta_\mu(k_j,f(k_j)).
   \]
   Since $f$ is $\beta$-approximating for $\mu$,
   \[
   \lim_{j\to\infty} \|z^{k}-z^{k+k_j}\pi_{k_j}\|_{L^2(\T,\mu)}=0.
   \]
   Thus, $z^{k}\in\overline{E(\Lambda(\mathbf{k},\bm{\ell}))}^{L^2(\T,\mu)}$ which implies that $E(\Lambda(\mathbf{k},\bm{\ell}))$ is dense in $L^2(\T,\mu)$. 
\end{proof}
\section{Applications}
We now want to prove Theorem~\ref{Existence of infinite sets}.
\begin{proof}[Proof of Theorem~\ref{Existence of infinite sets}]
Let $\mu\in\MMeas$ and let $1<t<s$. Put 
\[
f:\N\to\N,\quad f(n):=\lfloor n^{t}\rfloor.
\]
We first show that $f$ is $\beta$-approximating for $\mu$. By Lemma~\ref{LemmaMarkoffMeasure} there exist $\varepsilon>0$ and $\ell\in\N$ such that 
\[
\|\Phi_n\|^2_{L^2(\T,\mu)}\leq(1-\varepsilon^2)^{\frac{n-1}{\ell}-1}
\]
for all $n\in\N$. Thus, for every $k\in\N$,
    \begin{align*}
       \|\Phi_{\lfloor k^t\rfloor+1}\|_{L^2(\T,\mu)}(2\lfloor k^t\rfloor+2)^{k-1}&\leq (1-\varepsilon^2)^{\frac{\lfloor k^{t}\rfloor}{2\ell}-\frac{1}{2}}  (2\lfloor k^t\rfloor +2)^{k-1}\\
       &\leq (1-\varepsilon^2)^{\frac{ k^{t}}{2\ell}-\frac{1}{2}-\frac{1}{2\ell}}  (2k^t +2)^{k-1}\\
       &\leq\frac{(1-\varepsilon^2)^{-\frac{1}{2}-\frac{1}{2\ell}}}{2k^t+2}\Big((1-\varepsilon^2)^{\frac{ k^{t-1}}{2\ell}}(2k^{t}+2)\Big)^{k}.
   \end{align*}
   The right-hand side converges to 0 when $k\to \infty$. Thus, by Lemma~\ref{LemmaDistance2}, $f$ is $\beta$-approximating for $\mu$. Furthermore,
    \[
    \lim_{j\to\infty} \lfloor k_j^s\rfloor -f(k_j)=\lim_{j\to\infty} \lfloor k_j^s\rfloor -\lfloor k_j^t\rfloor=\infty
    \]
   since $t<s$. Hence, by Theorem~\ref{Theorem beta}, $E(\Lambda(\lfloor \mathbf{k},\lfloor \mathbf{k}^s\rfloor))$ is dense in $L^2(\T,\mu)$. Finally, if one chooses $\mathbf{k}$ such that for every $j\in\N$
   \[
   (k_j+\lfloor k_j^s\rfloor)j+1\leq k_{j+1}
   \]
   then
    \begin{align*}
    \underline{\dd}(\Lambda(\mathbf{k},\lfloor \mathbf{k}^s\rfloor))&\leq \liminf_{j\to\infty} \frac{\big|\Lambda(\mathbf{k},\lfloor \mathbf{k}^s\rfloor)\cap \llbracket 1,k_{j+1}-1\rrbracket\,\big|}{k_{j+1}-1}\\[1em]
    &\leq \liminf_{j\to\infty} \frac{k_j+\lfloor k_j^s\rfloor}{k_{j+1}-1}\leq \liminf_{j\to\infty}\frac{1}{j}=0.\qedhere
\end{align*}
\end{proof}
\begin{remark}
    Comparing the results from \cite{OlevskiiUlanovskii} with ours, one can have \linebreak${(\Lambda,\MMeas)\in\A}$ for sets $\Lambda$ that are much sparser than any known set $\Lambda$ for which one has ${(\Lambda,\mathcal{W}\cap\Sz^c)\in\A}$. To be more precise, in \cite{OlevskiiUlanovskii} it was proven that $(\N_0\setminus \Gamma,\mathcal{W}\cap\Sz^c)\in\A$ if
    \[
    \sum_{\gamma\in\Gamma} \frac{1}{\sqrt{\gamma}}<\infty.
    \]
    This implies $\underline{\dd}(\N_0\backslash\Gamma)=1$ since for every $M\in\N$ one has 
    \begin{align*}
       \infty> \sum_{\gamma\in\Gamma} \frac{1}{\sqrt{\gamma}}&\geq \limsup_{N\to\infty} |\Gamma\cap\llbracket1,N\rrbracket|\cdot \frac{1}{\sqrt{N}}\geq \overline{\dd}(\Gamma)\sqrt{M}.
    \end{align*}
    Thus, by letting $M\to\infty$, it follows that $\overline{\dd}(\Gamma)=1-\underline{\dd}(\N_0\backslash\Gamma)=0$.
    
    For the class $\MMeas$ by contrast, there exists $\Lambda\subseteq \N$ with $\underline{\dd}(\Lambda)=0$ such that $(\Lambda,\MMeas)\in\A$ by Theorem~\ref{Existence of infinite sets}.
\end{remark}
We now look at the class $\mathrm{Mar}_{\varepsilon,\ell}(\T)$ introduced in \eqref{markoffmeasure}. In this class one can find an even ‘thinner' set $\Lambda(\mathbf{k},\bm{\ell})$ than the one in Theorem~\ref{Existence of infinite sets} such that 
\[
\big(\Lambda(\mathbf{k},\bm{\ell}),\mathrm{Mar}_{\varepsilon,\ell}(\T)\big)\in\A.
\]
To be more precise, in the situation when $\mu\in\MMeas$ (Theorem~\ref{Existence of infinite sets}), there are $\beta$-approximating functions $f$ with
\[
f(k)=O(k^s),\quad (k\to\infty)
\]
for any $s>1$. 
However, when $\mu\in\mathrm{Mar}_{\varepsilon,\ell}(\T)$ (Corollary~\ref{Existence of parameter dependent infinite sets}), there exist $f$ which are $\beta$-approximating for $\mu$ with
\[
f(k)=O(\log(k)k),\quad (k\to\infty)
\]
 where the implicit constant depends on $\varepsilon$ and $
\ell$. 
\begin{corollary}\label{Existence of parameter dependent infinite sets}
    Let $\varepsilon>0, \ell\in\N$, $t>2$ and let $\mathbf{k}:=(k_j)_{j\in\N}\in\N^\N$ be strictly increasing. Put 
    \[
    C_{\varepsilon,\ell}:=\frac{-\ell}{\log(1-\varepsilon^2)}
    \]
    and $\lfloor t\,C_{\varepsilon,\ell}\log(\mathbf{k})\, \mathbf{k}\rfloor:=\big(\lfloor t\,C_{\varepsilon,\ell}\log(k_j)\, k_j\rfloor\big)_j$.
    Then, 
    \[
    \Big(\Lambda(\mathbf{k},\lfloor t\,C_{\varepsilon,\ell}\log(\mathbf{k})\, \mathbf{k}\rfloor),\, \mathrm{Mar}_{\varepsilon,\ell-1}(\T)\Big)\in\A.
    \]
\end{corollary}
\begin{proof}
Let $\mu\in\mathrm{Mar}_{\varepsilon,\ell-1}(\T)$ and let $2<\tau<t$. First, we show that 
\[
f:\N\to\N,\quad  f(k):=\lfloor \tau\,C_{\varepsilon,\ell}\log(k)\, k\rfloor
\]
is $\beta$-approximating for $\mu$.
By Lemma~\ref{LemmaMarkoffMeasure} b),
\[
\|\Phi_n\|^2_{L^2(\T,\mu)}\leq(1-
\varepsilon
^2)^{\frac{n-1}{\ell}-1}.
\]
Thus, we can estimate
\begin{align*}
       &\|\Phi_{\tau \lfloor C_{\varepsilon,\ell}\log(k)k\rfloor+1}\|_{L^2(\T,\mu)}(2\lfloor \tau C_{\varepsilon,\ell}\log(k)k\rfloor+2)^{k-1}\\[1ex]
       &\quad \quad\leq (1-\varepsilon^2)^{\frac{\lfloor \tau C_{\varepsilon,\ell}\log(k)k\rfloor}{2\ell}-\frac{1}{2}}  (2\lfloor \tau C_{\varepsilon,\ell}\log(k)k\rfloor+2)^{k-1}\\[1ex]
       &\quad \quad\leq (1-\varepsilon^2)^{\frac{\tau C_{\varepsilon,\ell}\log(k)k}{2\ell}-\frac{1}{2}-\frac{1}{2\ell}}  (2\tau C_{\varepsilon,\ell}\log(k)k +2)^{k-1}\\[1ex]
       &\quad \quad\leq\frac{(1-\varepsilon^2)^{-\frac{1}{2}-\frac{1}{2\ell}}}{2\tau\,C_{\varepsilon,\ell}\log(k) k+2}\cdot \Big((1-\varepsilon^2)^{\frac{\tau\,C_{\varepsilon,\ell}}{2\ell}\log(k) }\big(2\tau\,C_{\varepsilon,\ell}\log(k) k+2\big)\Big)^{k}\\
       &\quad \quad\leq \frac{(1-\varepsilon^2)^{-\frac{1}{2}-\frac{1}{2\ell}}}{2\tau\,C_{\varepsilon,\ell}\log(k) k+2}\cdot \Big(k^{-\frac{\tau}{2}}\big(2\tau\,C_{\varepsilon,\ell}\log(k) k+2\big)\Big)^{k}.
   \end{align*}
   The right-hand side converges to 0 when $k$ goes to infinity. Thus, by Lemma~\ref{LemmaDistance2}, $f$ is $\beta$-approximating for $\mu$.
    Furthermore, 
    \[
    \lim_{j\to\infty} \lfloor t\,C_{\varepsilon,\ell}\log(k_j)\, k_j\rfloor-f(k_j)=\lim_{j\to\infty} \lfloor t\,C_{\varepsilon,\ell}\log(k_j)\, k_j\rfloor-\lfloor \tau\,C_{\varepsilon,\ell}\log(k_j)\, k_j\rfloor=\infty
    \]
    since $\tau<t$. Thus, by Theorem~\ref{Theorem beta}, $E(\Lambda(\mathbf{k},\lfloor t\,C_{\varepsilon,\ell}\log(\mathbf{k})\, \mathbf{k}\rfloor))$ is dense in $L^2(\T,\mu)$.
\end{proof}
As a consequence of Theorem~\ref{Existence of infinite sets}, we show that if 
$\Lambda = \mathbb{N}_0 \setminus \Gamma$ with 
\[
\Gamma = \{\lfloor e^{t^n} \rfloor \mid n \in \mathbb{N}\}
\] 
for some $t > 1$, then $(\Lambda, \MMeas) \in \A$. One can in fact even choose $\Gamma$ to be a union of large intervals, in the sense that the intervals have double exponential length.
\begin{corollary}\label{ExampleGamma}
   Let $t\geq \Tilde{t}>1$ and $C>0$. Let
    \[
    \Gamma:=\bigcup_{j\in\N}\llbracket\lfloor e^{t^j}\rfloor,\lfloor e^{t^j}\rfloor+\lfloor Ce^{\Tilde{t}^j}\rfloor\rrbracket.
    \]
    Then $(N_0\backslash \Gamma,\MMeas)\in\A$.
\end{corollary}
\begin{proof}
    Let $1<s<t$. Then
    \begin{align*}
        \limsup_{j\to\infty} \frac{\lfloor e^{t^{j+1}
        }\rfloor}{(\lfloor e^{t^j}\rfloor+\lfloor Ce^{\Tilde{t}^j}\rfloor+1)^s}&\geq \limsup_{j\to\infty} \frac{e^{t^{j+1}}}{2( e^{t^j}+Ce^{\Tilde{t}^j}+1)^s}\geq \limsup_{j\to\infty} \frac{e^{t^{j+1}}}{4e^{st^j}}=\infty.
    \end{align*}
    Thus, there exists $N\in\N$ such that for every $j\geq N$
    \begin{equation*}
        \lfloor e^{t^{j+1}}\rfloor \geq 2\Big(\lfloor e^{t^{j}}\rfloor+\lfloor Ce^{\Tilde{t}^j}\rfloor +1\Big)^{s}.
    \end{equation*}
    Put $k_j:=\lfloor e^{t^{j}}\rfloor+\lfloor Ce^{\Tilde{t}^j}\rfloor +1$ for every $j\in\N$. It follows that
    \begin{align*}
        \N_0\backslash \Gamma \supseteq &\bigcup_{j\geq N} \llbracket k_j,\lfloor e^{t^{j+1}}\rfloor\rrbracket \supseteq \bigcup_{j\geq N} \llbracket k_j,k_j+\lfloor k_j\rfloor^s\rrbracket=\Lambda(\mathbf{k},\lfloor \mathbf{k}^s\rfloor)
    \end{align*}
    where $\mathbf{k}:=(k_j)_{j\geq N}$ and $\lfloor \mathbf{k}^s\rfloor:=(\lfloor k_j^s\rfloor)_{j\geq N}$. Since $E(\Lambda(\mathbf{k},\lfloor \mathbf{k}^s\rfloor))$ is dense in $L^2(\T,\mu)$ for every $\mu\in\MMeas$ by Theorem~\ref{Existence of infinite sets}, also $E(\N_0\setminus \Gamma)$ is dense in $L^2(\T,\mu)$ for every $\mu\in\MMeas$.
\end{proof}
\bibliography{main}

@article{KhrushchevClassifictation,
 author = {Khrushchev, S. V.},
 title = {Classification theorems for general orthogonal polynomials on the unit circle},
 fjournal = {Journal of Approximation Theory},
 journal = {J. Approx. Theory},
 issn = {0021-9045},
 volume = {116},
 number = {2},
 pages = {268--342},
 year = {2002}
}

@book{SimonOPUC1,
 author = {Simon, Barry},
 title = {Orthogonal polynomials on the unit circle. {Part} 1: {Classical} theory},
 fseries = {Colloquium Publications. American Mathematical Society},
 series = {Colloq. Publ., Am. Math. Soc.},
 issn = {0065-9258},
 volume = {54},
 isbn = {0-8218-3446-0},
 year = {2005},
 publisher = {Providence, RI: American Mathematical Society (AMS)}
}

@article{OlevskiiUlanovskii,
 author = {Olevskii, Alexander and Ulanovskii, Alexander},
 title = {On the {Szeg{\H{o}}}-{Kolmogorov} prediction theorem},
 fjournal = {Israel Journal of Mathematics},
 journal = {Isr. J. Math.},
 issn = {0021-2172},
 volume = {246},
 number = {1},
 pages = {335--351},
 year = {2021}
}

@article {Bingham12,
    AUTHOR = {Bingham, N. H.},
     TITLE = {Szeg{\H{o}}'s theorem and its probabilistic descendants},
   JOURNAL = {Probab. Surv.},
  FJOURNAL = {Probability Surveys},
    VOLUME = {9},
      YEAR = {2012},
     PAGES = {287--324},
      ISSN = {1549-5787}
}

@article {Kolmogorov41,
    AUTHOR = {Kolmogorov, A. N.},
     TITLE = {Stationary sequences in {H}ilbert's space},
   JOURNAL = {Bolletin Moskovskogo Gosudarstvenogo Universiteta. Matematika},
  FJOURNAL = {Bolletin Moskovskogo Gosudarstvenogo Universiteta. Matematika},
    VOLUME = {2},
      YEAR = {1941},
     PAGES = {1-40},
   MRCLASS = {46.0X},
    Note   = {(in Russian; reprinted, \textit{Selected works of A. N. Kolmogorov, Vol. 2: Theory of probability and
mathematical statistics}, Nauka, Moskva, 1986, 215-255).}
}

@book{SimonOPUC2,
 author = {Simon, Barry},
 title = {Orthogonal polynomials on the unit circle. {Part} 2: {Spectral} theory},
 fseries = {Colloquium Publications. American Mathematical Society},
 series = {Colloq. Publ., Am. Math. Soc.},
 issn = {0065-9258},
 volume = {54},
 isbn = {0-8218-3675-7},
 year = {2005},
 publisher = {Providence, RI: American Mathematical Society}
}

@article{Zhedanov20,
 author = {Zhedanov, Alexei},
 title = {An explicit example of polynomials orthogonal on the unit circle with a dense point spectrum generated by a geometric distribution},
 fjournal = {SIGMA. Symmetry, Integrability and Geometry: Methods and Applications},
 journal = {SIGMA, Symmetry Integrability Geom. Methods Appl.},
 issn = {1815-0659},
 volume = {16},
 pages = {paper 140, 9},
 year = {2020}
}

@misc{SzegoGrenander58,
 author = {Grenander, Ulf and Szeg{\H{o}}, Gabor},
 title = {Toeplitz forms and their applications},
 year = {1958},
 howpublished = {California {Monographs} in {Mathematical} {Sciences}. {Berkeley} and {Los} {Angeles}: {University} of {California} {Press}. viii, 245 p. (1958).}
}

@article{Rakhmanoc83,
 author = {Rakhmanov, E. A.},
 title = {On the asymptotics of the ratio of orthogonal polynomials. {II}},
 fjournal = {Mathematics of the USSR, Sbornik},
 journal = {Math. USSR, Sb.},
 issn = {0025-5734},
 volume = {46},
 pages = {105--117},
 year = {1983}
}

@article{Krein45,
 author = {Krein, M.},
 title = {On a generalization of some investigations of {G}. {Szeg{\H{o}}}, {V}. {Smirnoff} and {A}.~{N}.~{Kolmogoroff}},
 fjournal = {Comptes Rendus (Doklady) de l'Acad{\'e}mie des Sciences de l'URSS, Nouvelle S{\'e}rie},
 journal = {C. R. (Dokl.) Acad. Sci. URSS, n. Ser.},
 issn = {1819-0723},
 volume = {46},
 pages = {91--94},
 year = {1945}
}

@book{Carmona90,
 author = {Carmona, Ren{\'e} and Lacroix, Jean},
 title = {Spectral theory of random {Schr{\"o}dinger} operators},
 fseries = {Probability and its Applications},
 series = {Probab. Appl.},
 issn = {2297-0371},
 isbn = {0-8176-3486-X},
 year = {1990},
 publisher = {Basel etc.: Birkh{\"a}user Verlag}
}

@book{Pastur92,
 author = {Pastur, Leonid and Figotin, Alexander},
 title = {Spectra of random and almost-periodic operators},
 fseries = {Grundlehren der Mathematischen Wissenschaften},
 series = {Grundlehren Math. Wiss.},
 issn = {0072-7830},
 volume = {297},
 isbn = {3-540-50622-5},
 year = {1992},
 publisher = {Berlin etc.: Springer-Verlag},
 language = {English},
 zbMATH = {52135},
 Zbl = {0752.47002}
}

@article{Anderson58,
  title = {Absence of Diffusion in Certain Random Lattices},
  author = {Anderson, P. W.},
  journal = {Phys. Rev.},
  volume = {109},
  issue = {5},
  pages = {1492--1505},
  numpages = {0},
  year = {1958},
  publisher = {American Physical Society}
}
\end{document}